\documentclass[10pt]{amsart}
\usepackage{a4,amsmath,amssymb,amsthm,eucal}
\usepackage[english]{babel}
\usepackage{graphicx}
\usepackage{xcolor}
\usepackage{epsfig}

\newcommand\N{\mathbb{N}}
\newcommand\R{\mathbb{R}}

\newcommand{\forget}[1]{}

\def\Om{{\Omega}}  
\def\om2{{\Om\times\Om}}
\def\M{{\mathcal M}}

\def\supp{\mathrm{supp}\,}

\def\div{\mathrm{div}\,}

\def\Lip{\mathrm{Lip}}

\usepackage{color}

\newtheorem{theorem}{Theorem}[section]

\newtheorem{lemma}[theorem]{Lemma}
\newtheorem{proposition}[theorem]{Proposition}

\theoremstyle{remark}
\newtheorem{remark}[theorem]{Remark}

\numberwithin{equation}{section}

\begin{document}
\title{Constructive controllability for incompressible vector fields}
\author{Sergey Kryzhevich}
\address[Sergey Kryzhevich]{
Institute of Applied Mathematics, Faculty of Applied Physics and Mathematics, Gdańsk University of Technology, 80-233 Gdańsk, Poland
\and
BioTechMed Center, Gdańsk University of Technology 
\and 
Faculty of Mathematics and Computer Science, St. Petersburg State University, 13B Universitetskaya Emb.,
St.\ Petersburg 199034, Russia
}
\email[Sergey Kryzhevich]{kryzhevich@gmail.com}
\author{Eugene Stepanov}
\address[Eugene Stepanov]{
St.Petersburg Branch
of the Steklov Mathematical Institute of the Russian Academy of Sciences,
Fontanka 27,
191023 St.Petersburg,
Russia
\and
Higher School of Economics, Faculty of Mathematics, Usacheva str. 6, 119048 Moscow, Russia,
\and
Department of Mathematical Physics, Faculty of Mathematics and Mechanics,
St. Petersburg State University, Universitet\-skij pr.~28, Old Peterhof,
198504 St.Peters\-burg, Russia
}
\email[Eugene Stepanov]{stepanov.eugene@gmail.com}
\thanks{
The work of the first author was supported by Gda\'{n}sk University of Technology by the
DEC 14/2021/IDUB/I.1 grant under the Nobelium - ’Excellence Initiative - Research University’
program.
	The work of the second author
is partially supported by RFBR grant \#20-01-00630A
}

\date{December 1, 2021}

\begin{abstract}
	We give a constructive proof of a global controllability result for an autonomous system of ODEs guided by bounded 
	locally Lipschitz and divergence free (i.e.\ incompressible) vector field, when the phase space is the whole Euclidean space
	and the vector field satisfies so-called vanishing mean drift condition. For the case when the ODE is defined over some smooth compact connected Riemannian manifold, we  significantly strengthen the assertion of the known controllability theorem in absence of nonholonomic constraints by proving that one can find a control steering the state vector from one given point to another by using the observations of only the state vector, i.e., in other words, by changing slightly the vector field, and such a change can be made small not only in uniform, but also in Lipschitz (i.e. $C^1$) topology.
\end{abstract}

\keywords{global controllability, Poincar\'{e} recurrence theorem, Poisson stable points}

\maketitle

\section{Introduction}
Consider the ordinary differential equation (ODE)
\begin{equation}\label{eq_ode1}
\dot{x}=V(x)
\end{equation}
where $V$ is a globally bounded smooth (or at least locally Lipschitz) vector field in a Euclidean phase space $\R^d$ satisfying 
\textit{divergence free} (or else also called \textit{incompressibility}) condition
\[
\div V=0,
\]
where for locally Lipschitz vector fields the divergence operator is defined almost everywhere due to Rademacher’s theorem.
The classical point-to-point controllability problem is that of finding, given two points $p$ and $q$ in the phase space, the 
control function $u=u(t)$ steering the state vector $x(\cdot)$ from $p$ to $q$, i.e.\, formally, such that
	the solution (trajectory) of the ODE
\begin{equation}\label{eq_odectrl2}
\dot x = V(x)+u(t),
\end{equation}
starting at $x(0)=p$ satisfies $x(T)=q$ for some $T>0$. 
Usually the control is required to be small.
The problem has been interpreted in~\cite{BurNovIv16-fish} in the following way: 
a fish in an unbounded turbulent ocean (modelled by the phase space $\R^d$) with the flow velocity field given by $V$ is
able to move with its own velocity $u$ not exceeding in modulus the given value $\varepsilon>0$. 
The assumptions of boundedness and incopressibility of $V$ are quite natural in this setting. The control problem termed in these words is that of asking whether the fish can reach any point starting from an arbitrary one. If the answer to this question is positive, the ODE~\eqref{eq_ode1} is called \textit{globally controllable}.
 	
 If the phase space of~\eqref{eq_odectrl2} is a smooth compact manifold instead of $\R^d$, then the answer to the posed question is positive and provided by the known global controllability result (theorem~4.2.7 in~\cite{Bloch15-control} where it is formulated for analytic vector fields
 on compact Riemannian manifolds). Nevertheless in the whole $\R^d$
 the incompressibility condition of $V$ is not enough for global controllability to hold as can be seen just taking $V$ to be constant vector field with sufficiently large norm. 
 However,  it has been proven in~\cite{BurNovIv16-fish} that if $V$ is incompressible and has \emph{vanishing mean drift}
 (called small mean drift in~\cite{BurNovIv16-fish}) in the sense
 \begin{equation}\label{eq_VMD1}
 \lim_{\ell\to \infty} \sup_{x\in \R^d}
 \left|\frac{1}{\ell^d}\int_{[0,\ell]^d} V(x+y)\, dy\right|=0,
 \end{equation}
 then~\eqref{eq_ode1} is globally controllable (this result has been further extended in~\cite{BurNovIv17-fish} to nonautonomous ODEs).
 	Roughly speaking, the assumption of vanishing mean drift means that the average value of the flow velocity over big boxes vanishes with the corresponding limit uniform with respect to the selection of those big boxes.
 
 In~\cite{KryzhSte19} we took a completely different approach to the proof of this global controllability result obtaining, as a side product, a version of C.~Pugh's closing lemma for divergence free vector fields on the whole $\R^d$ satisfying vanishing mean drift condition~\eqref{eq_VMD1}. However, both our proof of the controllability result and the original proof from~\cite{BurNovIv16-fish} are inherently nonconstructive: in other words they assure the fish that  it can reach any desired destination but do not give any clue of how to do it. The goal of this paper is to provide a constructive proof of the global controllability result for~\eqref{eq_ode1}, namely, providing not only the result itself but an explicit construction of
 the steering control $u(\cdot)$. Our basic instruments 
 will be the main result of~\cite{KryzhSte19} which gives a constructive way of changing slightly the vector field so that all the points of the phase space become nonwandering for the corrected vector field.
 
 We also compare the proven controllability result over the whole $\R^d$ with the original classical setting of the global controllability theorem, 
 i.e.\ with the case when  the ODE~\eqref{eq_ode1} is defined over some smooth compact connected Riemannian manifold without boundary instead of $\R^d$. In this case we  significantly strengthen its assertion, by proving that one can find a control $u(\cdot)$ in the form $u(t)= W(x(t))$, where $x(\cdot)$ is the trajectory of the controlled system (i.e.\ we may control the system by using the observations of only its state vector), and the vector field $W$ is small not only in uniform, but also in Lipschitz (i.e. $C^1$) norm. 
In other words, we give a constructive proof for the Connecting Lemma for orbits in the case of divergence-free vector fields. 

\section{Notation and preliminaries}


The Euclidean norm in the finite-dimensional space $\R^d$ will be denoted by $|\cdot|$,
 $B_r(x)\subset \R^d$ stands for the open Euclidean ball of radius
$r$ centered at $x$, and $x\cdot y$ stands
for the usual scalar product of $x\in \R^d$ and $y\in \R^d$, and $\mathcal{L}^d$ stands for tor the $d$-dimensional Lebesgue measure. 
For any set $D\subset \R^d$, we let $\bar D$ be its closure.
By 
$\Lip(\R^d;\R^d)$ (resp.\ $\Lip_{loc}(\R^d;\R^d)$, $C^1(\R^d;\R^d)$)
we denote
the set of  
Lipschitz (resp.\ locally Lipschitz, continuously differentiable)
functions $f\colon \R^d\to \R^d$. 
The standard uniform norm of functions and vector functions will be denoted by $\|\cdot\|_\infty$.
For a $V\in \Lip(\R^d;\R^d)$ we denote by $\Lip\, V$ its least Lipschitz constant, and set
$\|V\|_{\Lip}:=\|V\|_\infty +\Lip\, V$. The notation $C_0^\infty(U)$ stands for the class of infinitely differentiable real-valued functions with compact support
in an open $U\subset \R^d$. 

For a diffeomorphism $f\colon U\subset \R^d\to  f(U)\subset \R^d$, where $U\subset \R^d$ is open, and a vector field $V$ over $U$
we denote the push-forward $f_* V$ of $V$ by $f$ by the formula
\[
(f_* V)(y):= (Df)(f^{-1}(y)) V (f^{-1}(y))=(Df^{-1})^{-1}(y) V (f^{-1}(y)),
\] 
so that if $x(\cdot)$ is a trajectory of the ODE $\dot{x}=V(x)$ in $U$, then  $y(\cdot):=f( x(\cdot))$ is a trajectory of
the ODE $\dot{y}=(f_*V)(x)$ in $f(U)$. Of course an identical (up to notation) definition can be done for vector fields on smooth manifolds and diffeomorphisms of open subsets of smooth manifolds. For a distribution $u$ (in particular, a measure) over $U$ we define its pushforward $f_{\#}u$ by $f$ over $f(U)$
as
\[
\langle\varphi, f_{\#}u\rangle : = \langle\varphi\circ f, u\rangle
\]
for every test function $\varphi \in C_0^\infty(\R^d)$, where $\langle \varphi, v\rangle $ stands for the action of a distribution
$v$ on a test function $\varphi$, and $\varphi\circ f$ stands for the composition of $\varphi$ with $f$, once $f$ is sufficiently smooth (e.g. when $u$ is a finite Borel measure, then this definition can be extended to $\varphi$ just bounded and continuous and $f$ just Borel). 

\section{Global controllability}

The following theorem is the main result of this paper. 

\begin{theorem}\label{th_control1constr}
	Let $V\in C^1(\R^d;\R^d)\cap \Lip(\R^d;\R^d)$ be a bounded incompressible vector field
	with uniformly continuous first derivatives, and  satisfy vanishing mean drift condition~\eqref{eq_VMD1}. Then for every couple of points $\{p,q\}\subset\R^d$ and every $\varepsilon>0$ there is a piecewise continuous function $u\colon \R^+\to\R^d$ (``control'') with $\|u\|_\infty <\varepsilon$ such that
	the trajectory of the ODE~\eqref{eq_odectrl2}
	satisfying $x(0)=p$ passes through $q$, i.e.\ $x(T)=q$ for some $T>0$.
\end{theorem}

The rest of this section will be dedicated to the proof of the above Theorem~\ref{th_control1constr}.

The following statement uses a construction from proposition~3.3 of~\cite{KryzhSte19} and is of some independent interest.

\begin{lemma}\label{lm_trans2}
	Let $x\colon [a,s]\to \R^d$ be a trajectory of
	the ODE
	\[
	\dot{x}(t)= F(t,x)
	\]
	over some time interval $[a,s]$, where $F\colon \R\times \R^d\to \R^d$ is bounded and continuous. 
	Then for every $\varepsilon>0$ there is a $\tau\in (0,s-a)$ and a $\rho>0$ depending on $\varepsilon$ and $\tau$ such that
	whenever $|x(s)-y|< \rho$,  there is a piecewise continuous control $u_\varepsilon(\cdot)$ with $|u_\varepsilon(\cdot)| <\varepsilon$ different from zero only on $(s-\tau, s]$,  for which 
	there is a solution 
	$x_\varepsilon\colon [a,s]\to \R^d$ of
	the ODE
	\begin{equation}\label{eq_Xdel1}
	\dot{x}_\varepsilon(t)= F(t,x_\varepsilon(t)) +u_\varepsilon(t)
	\end{equation}	
	coinciding with $x(\cdot)$ over  
	$[a,s-\tau]$ and having
	$x_\varepsilon(s)=y$.
\end{lemma}

\begin{remark}\label{rm_trans2a}
Note that in Lemma~\ref{lm_trans2} one can choose $\tau$ arbitrarily small, but the smaller is $\tau$, the smaller becomes also $\rho$.
In particular, it is easily deduced from the proof that when 
\[
|F(t,\bar z)-F(t,z)|\leq L |\bar z-z|
\]
for some $L>0$, then one can take 
\begin{equation}\label{eq_taurho1}
\tau< \min\left(s-a, \varepsilon/4L,  \varepsilon/8L \|F\|_\infty\right), \quad \rho=\tau\varepsilon/4.
\end{equation}
\end{remark}

\begin{proof}
	Denote $z:=x(s)$. Given a $\varepsilon>0$, choose a $\delta>0$ depending on $\varepsilon$ so small that
	\begin{equation}\label{eq_SFcorr1}
	|F(t,\bar z)-F(t,z)|< \varepsilon/4,
	\end{equation}
	for all $\bar z\in B_\delta(z)$ and all $t\in [a,s]$, and a $\tau\in (0,s-a)$ (depending on $\delta$, hence on $\varepsilon$) so small that
	\begin{eqnarray}
	\label{eq_SFcorr3}
	\|F\|_\infty\tau< \delta/2,
	\end{eqnarray}
	so that in particular
	\begin{eqnarray}
	\label{eq_SFcorr4}
	|x(s-\tau)-z|\leq \|F\|_\infty\tau <\delta/2.
	\end{eqnarray}
	
	Denote $\bar x(\cdot)$ over $[s-\tau, s]$ the trajectory
	of the ODE
	\[
	\dot{\bar x}(t)= F(t,z).
	\]
	satisfying
	$\bar x(s-\tau)= x(s-\tau)$.
	We get that both
	$\bar x(t)\in B_\delta(z)$ and $x(t)\in B_\delta(z)$ for all
	$t\in [s-\tau, s]$ due to~\eqref{eq_SFcorr3}  and~\eqref{eq_SFcorr4}.
	Moreover, one has
	\begin{equation}\label{eq_SFcorr1a}
	|z-\bar x(s)|=|x(s)-\bar x(s)| < \tau \varepsilon/4
	\end{equation}
	in view of~\eqref{eq_SFcorr1}.

	Fixed an arbitrary $y\in B_\rho(z)$ with a $\rho>0$ to be chosen later,
	we set 
	\[
	\alpha := (y- \bar x (s))/\tau\in \R^d.
	\]
	One has
	\begin{align*}
	|y-\bar x(s)| &\leq |y-z| + |z-\bar x(s)|\\
	& \leq |y-z| + \tau \varepsilon/4 \quad\mbox{by~\eqref{eq_SFcorr1a}}\\
	&< \rho + \tau \varepsilon/4,
	\end{align*}
	so that once we choose
	\[
	\rho:=\tau \varepsilon/4,
	\]
	we get
	\[
	|y-\bar x(s)|< \tau \varepsilon/2, 
	\]
	and hence
	\[
	|\alpha| < \varepsilon/2.
	\]
	
	Define
	$x_\varepsilon(t)$
	by 
	\[
	x_\varepsilon(t) := 
	\left\{
	\begin{array}{rl}
	x(t), & t\in [a, s-\tau],\\
	\bar x(t) +\alpha (t-s+\tau), & t\in (s-\tau,s].
	\end{array}
	\right. . 
	\]
	Note that $\dot{x}_\varepsilon(t)= F(t,z)+\alpha$ for $t\in (\tau-s,s)$, and hence $x_\varepsilon$ satisfies~\eqref{eq_Xdel1} with
	\[
	u_\varepsilon(t):=
	\left\{
	\begin{array}{rl}
	0,& t\in [a, s-\tau],\\
	F(t,z)- F(t,x_\varepsilon(t)) +\alpha, & t\in (s-\tau, s].
	\end{array}
	\right.
	\]
	Further, by the choice of $\alpha$ one has $x(s)=y$. 
	Finally, 
	\begin{align*}
	|u_\alpha(t)|& \leq |\alpha|+ |F(t,z)- F(t,x_\varepsilon(t))|\\
	& <
	\frac{\varepsilon}{2} + \sup_{\bar z \in B_\delta(z)} |F(t,\bar z)-F(t,z)|\\
	& < \frac{\varepsilon}{2} + \frac{\varepsilon}{4} 
	\quad\mbox{by~\eqref{eq_SFcorr1}}
	\\
	& < \varepsilon,
	\end{align*}
	which proves the claim.
\end{proof}

The following statement summarizes the results from~\cite{KryzhSte19}.

\begin{proposition}\label{prop_Wgensmall1}
 Suppose that $V\in \Lip_{loc}(\R^d;\R^d)$
	be a bounded incompressible
	vector field with vanishing mean drift.
	Then, given an $\varepsilon>0$, there is a bounded vector field
	$\tilde V\in \Lip_{loc}(\R^d;\R^d) $
such that	
\begin{itemize}
	\item[(i)] every point $x\in\R^d$ is nonwandering for the ODE
\begin{equation}\label{eq_tildeODE1}
	\dot{x}=\tilde V(x).
\end{equation}
	Moreover $\mathcal{L}^d$-a.e.\ point $x'\in\R^d$ is Poisson stable, i.e.\ 
	for the trajectory $x(\cdot)$ of~\eqref{eq_tildeODE1} with $x(0)=x'$ 
	there are sequences $\{t_k^\pm\}_k\subset \R$ with $\lim_k t_k^\pm =\pm\infty$
	such that
	\[
	x'=\lim_{k\to \infty} x(t_k^\pm).
	\]  
	\item[(ii)] $\|\tilde V - V \|_\infty < \varepsilon$, and, moreover, if $V\in C^1(\R^d;\R^d)\cap \Lip(\R^d;\R^d)$
	and has uniformly continuous first derivatives, then one can assume
	$\|\tilde V - V \|_{\Lip} < \varepsilon$,
	\item[(iii)] $\|\mathrm{div}\, \tilde V \|_\infty  < \varepsilon$,
	\item[(iv)] for some $p\in ((d-1)/2, d/2)$, and  
	$\alpha> \bar\alpha=\bar\alpha(p,\varepsilon)$ one has $\mathrm{div}\, \psi \tilde V  =0$, where
	$\psi(x):=(|x|^2 +\alpha^2)^{-p}$.
\end{itemize}
Moreover, every vector field $\tilde V\in \Lip_{loc}(\R^d;\R^d)$ satisfying~(iv), necessarily satisfies~(i).
\end{proposition}

\begin{proof}
	Part of claim~(i) (all points are nonwandering) and claims (ii)-(iv) are a summary of lemma~4.7 and theorem~4.8 from~\cite{KryzhSte19}. 
	Note that in theorem~4.8 of~\cite{KryzhSte19} it has been proven in fact, that for
	every vector field $\tilde V\in \Lip_{loc}(\R^d;\R^d)$ satisfying~(iv) the thesis of
	the Poincar\'{e} recurrence theorem as formulated in corollary~4.5 of~\cite{KryzhSte19}
	holds, which in particular means (see the proof of~\cite[corollary~4.5]{KryzhSte19} 
	or alternatively of~\cite[proposition 4.1.18]{KatHass95}) that
	claim~(i) is fully satisfied, i.e. not only all points are nonwandering but also $\mathcal{L}^d$-a.e. point is Poisson stable (note that in corollary~4.5 of~\cite{KryzhSte19} one only speaks of Poisson stability for positive semitrajectories; the proof for negative semitrajectories is completely symmetric).
\end{proof}

We are now able to prove Theorem~\ref{th_control1constr}.

\begin{proof}[Proof of Theorem~\ref{th_control1constr}]
Let $\tilde V$ be as in Proposition~\ref{prop_Wgensmall1} with $\varepsilon/3$ instead of $\varepsilon$,
i.e.\
\[
\|\tilde V - V \|_\Lip < \varepsilon
\]
and 
$\mathcal{L}^d$-a.e.\ point $x\in\R^d$ is Poisson stable
for the equation
\begin{equation}\label{eq_tildeODE3}
\dot{x}=\tilde V(x).
\end{equation}
In particular one has
\begin{equation}\label{eq_LipVtild1}
\Lip \tilde V\leq \Lip V +\varepsilon, \quad \Lip \|\tilde V\|_\infty \leq \|V\|_\infty +\varepsilon.
\end{equation}

Choose a $\rho>0$ such that
\begin{equation}\label{eq_taurho2}
\rho <\min \left(\frac{1}{4}, \frac{\varepsilon^2}{144(\Lip V +\varepsilon)}, \frac{\varepsilon^2}{288(\Lip V +\varepsilon)(\|V\|_\infty +\varepsilon)} \right), \tau=\rho\varepsilon/12.  
\end{equation}
Choose
$\{x_{j}\}_{j=1}^n\subset \R^d$ such that
\[
x_{1}=p, x_{n}=q, \quad |x_{j}-x_{j+1}|< \rho/4,\quad  j=1,\ldots, n-1.
\]  
Let also
\[
x_{n}':=x_n=q
\]
For an arbitrary $\delta\in (0,\rho/8)$ there are Poisson stable points
$\{x_{j}'\}_{j=1}^{n-1}\subset \R^d$ for the equation~\eqref{eq_tildeODE3} satisfying
$x_j'\in B_\delta(x_j)$, hence in particular
\[
|x_{j}'-x_{j+1}'| \leq  |x_{j}-x_{j+1}| +2\delta < \rho/4 + 2\rho/8=\rho/2,\quad  j=1,\ldots, n-1.
\]  
Let $T_j>3/\varepsilon$ be such that a trajectory $x_j(\cdot)$ of~\eqref{eq_tildeODE3} with $x_j(0)=x_j'$ satisfies
\[
|x_j(T_j)-x_j'|\leq \rho/2, \quad j=1,\ldots, n-1.
\]
It is worth observing that in particular~\eqref{eq_taurho2} and~\eqref{eq_LipVtild1} imply
\begin{equation}\label{eq_taurho3}
\rho <\min \left(\frac{T_j (\varepsilon/3)}{4}, \frac{(\varepsilon/3)^2}{16\Lip \tilde V}, \frac{(\varepsilon/3)^2}{32\Lip \tilde V \|\tilde V\|_\infty} \right), \tau=\rho\frac{(\varepsilon/3)}{4}.  
\end{equation}

Construct $u_k=u_k(t)$, $k=1,\ldots, n$, with $\|u_k\|_\infty\leq \varepsilon/3$ inductively. Let $u_0:=0$.
Note that the trajectory of~\eqref{eq_tildeODE3}
starting at $x_1'$
arrives at $x_1(T_1)$ at $\tau_1 = T_1 > 0$.
Lemma~\ref{lm_trans2} applied with $a:= 0$, $s:= \tau_{1}$, $\varepsilon/3$ in place of $\varepsilon$, 
$F(t,x):= \tilde V(x)$, $y:=x_{2}'$
provides the existence of a control $u_{1}$ 
over $[0,\tau_{1}]$ which is nonzero only over $[\tau_1-\tau, \tau_1]$ such that the trajectory of
\begin{align*}
\dot{x}(t)&= \tilde V(x(t)) +u_0(t) = \tilde V(x(t)), 
\end{align*}
arrives at $x_2'$ at $\tau_1$.
Assume now that for some $k\in \N$ the function $u_k= u_k(t)$ be a control defined over $[0,\tau_k]$ such that the trajectory
of 
\begin{align*}
\dot{x}(t)&= \tilde V(x(t)) +u_k(t)
\end{align*}
starting at $x_1',$ 
arrives at $x_k'$ at some $\tau_k>0$, i.e.\ $x(\tau_k)=x_k$.
Let $u_{k+1}$ be the control coinciding with $u_k$ over $[0,\tau_k)$ and
over $[\tau_k,\tau_{k+1}]$, where $\tau_{k+1}:=\tau_k+T_k$, being the control provided  
by Lemma~\ref{lm_trans2} with $a:= \tau_k$, $s:= \tau_{k+1}$, $\varepsilon/3$ in place of $\varepsilon$, 
$F(t,x):= \bar V(x)$, $y:=x_{k+1}'$, i.e.\ is nonzero only over $[\tau_{k+1}-\tau, \tau_{k+1}]$ 
Note that the application of Lemma~\ref{lm_trans2}  is possible because~\eqref{eq_taurho3} gives exactly~\eqref{eq_taurho1} with these data. 
 
Proceeding in this way we get that that the trajectory $x(\cdot)$ of the ODE
\begin{align*} 
\dot{x}(t)&= \tilde V(x(t)) +u_n(t),
\end{align*} 
starting at $x(0) = x_1'$, eventually arrives at $x_n'=q$ at some instance $T:=\tau_n>0$, i.e.\ 
$x(T)=x_n'=q$. 

Now, since $\delta>0$ can be chosen arbitrarily small, then one can find by Remark~\ref{rm_tarj1b}
some $\bar V\in \Lip_{loc}(\R^d;\R^d)$ satisfying  $\|\bar V - \tilde V \|_\infty < \varepsilon/3$,
and the trajectory $\bar x(\cdot)$ of the ODE 
\begin{align*}
\dot{x}(t)&= \bar V(x(t)) +u_n(t) 
\end{align*}
satisfying $\tilde x(0)=p$ 
passes through $q$.
It suffices to define now
\[
u(t):= \bar V (\bar x(t))- V (\bar x(t)) + u_n(t),
\] 
and estimate
\[
\|u\|_\infty:= \|\bar V- V \|_\infty + \|u_n\|_\infty \leq \|\bar V- \tilde V \|_\infty +\|\tilde V- V \|_\infty + \|u_n\|_\infty < \varepsilon/3 +\varepsilon/3 + \varepsilon/3= \varepsilon
\]
by construction, to conclude the proof.
\end{proof}

\section{Case of a compact manifold}

It is worth comparing the Theorem~\ref{th_control1constr} with the case when the vector field $V$ and hence the differential
equation~\eqref{eq_ode1} are defined not on $\R^d$ but rather over some compact smooth connected Riemannian manifold $M$. 
In this case we can say more, namely, that one can find a control $u(\cdot)$ in the form $u(t)= W(x(t))$, where $x(\cdot)$ is the trajectory of the controlled system (i.e.\ in more control theoretic terminology we may control the system by observing only its state vector), and the vector field $W$ is small not only in uniform, but also in Lipschitz (i.e. $C^1$) norm.
Here for $C^1$ vector field $V$ on $M$ we denote
\[
\|V\|_{\Lip} :=\sup_{x\in M} |V(x))|_x + \sup_{\{(x,y)\in \M\times M, x\neq y\}} \frac{|P_{\gamma,x,y} V(x)-V(y) |_y}{d_M(x,y)},
\]
where $d_M$ stands for the Riemannian distance, $P_{\gamma,x,y}$ stands for the canonical parallel transport operator of a vector
in the tangent space $T_x M$ to the tangent space $T_y M$ along the geodesics $\gamma$, $|\cdot|_x$ stands for the norm over $T_xM$  provided by the metric tensor. Clearly the convergence of vector fields with respect to $\|\cdot\|_{\Lip}$ is equivalent to convergence in Whitney $C^1$ topology once $M$ is compact (this can be easily seen e.g. once one uses the rather simple and general equivalent definition of Whitney topologies from~\cite{DeFariaHaz20_whitneyCr}, see also~\cite[theorem 5.7]{Artigue15}).

\begin{theorem}\label{th_control1constr_comp}
	Let $M$ be a $C^\infty$ smooth compact connected Riemannian manifold without boundary, 
	the vector field $V$ over $M$ be $C^1$, i.e. be a continuously differentiable section of the tangent bundle $TM$ of $M$,
	satisfying $\div V=0$, the divergence being intended with respect to the volume measure of $M$.
	Then for every couple of points $\{p,q\}\subset\R^d$ and every $\varepsilon>0$ there is a $C^1$ vector field
	$\tilde V$ on $M$ such that
	\begin{equation}\label{eq_tildeVest2b1}
	\|\tilde V- V\|_{\Lip} < \varepsilon
	\end{equation}
	and 
	the trajectory of the ODE
	\begin{equation}\label{eq_odeMcomp1}
	\dot{x} = \tilde V(x)
	\end{equation}
	satisfying $x(0)=p$ passes through $q$, i.e.\ $x(T)=q$ for some $T>0$.
\end{theorem}

\begin{proof}
	It suffices to prove the claim for $p\neq q$.
	We use first theorem~1.1 from~\cite{Bessa08} to get a $C^1$ vector field
	$\hat V$ on $M$ such that
	\begin{equation*}\label{eq_tildeVest2b2}
	\|\hat V- V\|_{\Lip} < \varepsilon/2
	\end{equation*}
	and the flow of $\hat V$ is topologically mixing, i.e. for every couple of open sets $U_p$ and $U_q$ in $M$ and for all sufficiently large $T>0$ there is a trajectory of
	\[
	\dot{y}=\hat V(y)
	\]
	starting at a point $y(0)\in U_p$ and arriving at $y(T)\in U_q$.
	We use then Lemma~\ref{lm_corn1pt_manif2} below with $\hat V$ instead of $V$ and $\varepsilon/2$ instead of $\varepsilon$ 
	to get a $\tilde V$ such that the trajectory of~\eqref{eq_odeMcomp1} satisfying $x(0)=p$ passes through $q$, and
	\begin{equation*}\label{eq_tildeVest2b3}
\|\hat V- \tilde V\|_{\Lip} < \varepsilon/2,
\end{equation*}
which together with~\eqref{eq_tildeVest2b2} gives~\eqref{eq_tildeVest2b1} as claimed.
\end{proof}

The following statements have been used in the above proof.

\begin{lemma}\label{lm_corn1pt_manif2}
	Let $V$ be a Lipschitz vector field on a smooth compact connected Riemannian manifold $M$, which is topologically mixing.
	Then for every $p\in M$ and $q\in M$, $p\neq q$, disjoint open neighborhoods $U_p$ of $p$ and $U_q$ of $q$ in $M$ respectively, and $\varepsilon>0$ there is a  Lipschitz vector field $\tilde V$ on $M$ coinciding with $V$ outside of $U_p\cup U_q$ such that
	\begin{equation}\label{eq_tildeVest2b4}
	\|\tilde V- V\|_{\Lip} < \varepsilon
	\end{equation}
	and the trajectory $y(\cdot)$ over $M$ of   
	\begin{equation*}\label{eq_ode4a}
	\dot y= \tilde V (y), 
	\end{equation*}
	starting at $y(0)=p$
 passes through $q$, i.e.\ $y(T)=q$ for some $T>0$.	
\end{lemma}

\begin{remark}
It is easy to observe from the proof of the above Lemma~\ref{lm_corn1pt_manif2} that it holds under somewhat
milder assumption of $V$ than topological mixing property: in fact, transitivity (existence of a dense trajectory) would suffice.
\end{remark}

\begin{proof}
	Take a small $\rho>0$ such that there is a smooth embedding $f$ of the ball $B_\rho(p)\subset U_p\subset M$ into $\R^d$ 
	and of the ball $B_\rho(q)\subset U_q\subset M$
	(a smooth diffeomorphism onto the image) with 
\begin{eqnarray}
\label{eq_flip1a}
\frac 1 2 d(x,y )\leq |f(x)-f(y)|\leq 2 d(x,y),\quad\text{and}\\
\label{eq_flip1b}
||| f^{-1}_*||| \leq 2, 
\end{eqnarray}	
where $d$ stands for the distance in $M$
and  $|||f^{-1}_* |||$ stands for the norm of pushforward operator $f^{-1}_*$ of vector fields
seen as a linear operator between Lipschitz vector fields
in $\R^d$ and Lipschitz vector fields over $\bar B_\rho(p)$.  We consider the pushforward $f_* V$ of the vector field $V$ by $f$ defined over $f(B_\rho(p))\cup f(B_\rho(q))$.
Given an $\varepsilon>0$,  by Lemma~\ref{lm_corrTraj1} and Remark~\ref{rm_tarj1a} there is a 
$\delta\in (0,1)$ which we may take to satisfy $\delta <\rho/6$,
 with the following property:
whenever $x_1$ and $x_2$ are points in $M$, for which
\begin{equation}\label{eq_delta3a}
|f(p)-f(x_0)|\leq \delta^3,\quad |f(p)-f(x_1)|\leq \delta^3,
\end{equation}
and $\hat{x}(\cdot)$ a trajectory of 
\[
\dot{\hat x}= (f_*V) (\hat x), 
\]
satisfying  $\hat x (0) = f(x_1)$, $\hat x(T) = f(x_2)$ for some $T>0$,
then there is a vector field $\hat V\colon \R^d\to\R^d$ coinciding with $f_*V$ outside of $B_{2\delta}(f(x_1))\cup B_{2\delta}(f(x_2))$, such that
\begin{equation}\label{eq_tildeVest2a}
\|f_*V-\hat V\|_{\Lip} < \varepsilon/2
\end{equation}
and for the
trajectories $\hat y_1(\cdot)$ and $\hat y_2(\cdot)$ of
\begin{equation}\label{eq_ode2a}
\dot{y}=\hat V(y)
\end{equation}
satisfying $\hat y_1(0)=f(p)$ and $\hat y_2(T)=f(q)$
one has that  
\begin{itemize}
	\item[(i)] $\hat y_1(\cdot)$ coincides with  $\hat{x}(\cdot)$ over $f(B_\rho(p))\setminus B_{2\delta}(f(x_1))$ 
	\item[(ii)] $\hat y_2(\cdot)$ coincides with  $\hat{x}(\cdot)$ over $f(B_\rho(q))\setminus B_{2\delta}(f(x_2))$.
\end{itemize}

To prove the existence of a $\hat{x}$ and $T$ as above 
consider a trajectory $x(\cdot)$ of~\eqref{eq_ode1} such that $x_1:=x(0) \in B_{\delta^3/ 2}(p)$ and
$x_2:=x(T)\in B_{\delta^3/2}(q)$ for some $T>0$ (such a trajectory exists since $V$ is assumed to be topologically mixing), and define $\hat{x}(\cdot):= f(x(\cdot))$. Note that then~\eqref{eq_delta3a} are satisfied in view of~\eqref{eq_flip1a}.

Let now $\tilde V := f^{-1}_* \hat V$ stand for the pushforward of $V$ by $f^{-1}$,
and denote $y_i(\cdot):= f^{-1}(\hat y_i(\cdot))$, $i=1,2$. Then
\begin{itemize}
	\item[(i')] $y_1(\cdot)$ coincides with  $x(\cdot)$ over $B_\rho(p)\setminus f^{-1}(B_{2\delta}(f(x_1)))$, 
	\item[(ii')] $y_2(\cdot)$ coincides with  $x(\cdot)$ over $B_\rho(q)\setminus f^{-1}( B_{2\delta}(f(x_2)))$,
	\item[(iii')] $\tilde V$ coincides with $V$ over $(B_\rho(p)\cup B_\rho(q))\setminus f^{-1}(B_{2\delta}(f(x_1))\cup B_{2\delta}(f(x_2)))$.
\end{itemize}
Define $\tilde V$ over $M\setminus (B_\rho(p)\cup B_\rho(q))$ by setting $\tilde V(x):=V(x)$ for
$x\not\in B_\rho(p)\cup B_\rho(q))$. 
Note that
\begin{align}
f^{-1}(B_{2\delta}(f(x_1))) &\subset B_{4\delta}(f(x_1))\quad\text{by~\eqref{eq_flip1a}}\\
& \subset B_{2\delta^3 + 4\delta}(p)\quad\text{by~\eqref{eq_delta3a} and~\eqref{eq_flip1a}}\\
& \subset B_{6\delta}(p) \quad\text{because $0<\delta<1$},
\end{align}
and analogously $f^{-1}(B_{2\delta}(f(x_2)))  B_{6\delta}(q)$, and therefore
\[
B_\rho(p)\setminus f^{-1}(B_{2\delta}(f(x_1))) \supset B_\rho(p)\setminus B_{6\delta}(q) \neq \emptyset,
\]
and analogously
$B_\rho(q)\setminus f^{-1}(B_{2\delta}(f(x_2))) \neq \emptyset$. Thus the vector field $\tilde V$ defined over $M$ is smooth.
Moreover,~(i') and ~(ii') above imply that  in fact
$y_1$ and $y_2$ are the same trajectory $y$ of~\eqref{eq_ode2a}, thus satisfying $y(0)=p$, $y(T)=q$.
Finally, from~\eqref{eq_tildeVest2a} and~\eqref{eq_flip1b} we get~\eqref{eq_tildeVest2b4} concluding the proof.
\end{proof}

The following results have been used in the above proof.

\begin{lemma}\label{lm_corrTraj1}
For every $\varepsilon>0$ there is a $\delta\in (0,1)$ with the following property:
if $\theta$ is a trajectory of~\eqref{eq_ode1}, i.e.\
\[
\theta:=\{x(t)\colon t\in \R\},
\] 
where 
$x(\cdot)$ 
solves~\eqref{eq_ode1}, then for every $y_0\in \R^d$ such that
\[
|y_0-x(0)|\leq \delta^3
\]
there is a vector field $\tilde V\colon \R^d\to\R^d$ coinciding with $V$ outside of $B_{2\delta}(x(0))$ with
\begin{equation}\label{eq_tildeVest1}
\|V-\tilde V\|_{\Lip} < \varepsilon
\end{equation}
such that the (forward)
trajectory (i.e.\ positive semitrajectory) $y(\cdot)$ of
\begin{equation}\label{eq_ode2}
\dot{y}=\tilde V(y)
\end{equation}
satisfying $y(0) = y_0$ coincides with $x(\cdot)$ over the
segment $[t_1, t_2]$, where
\[
t_1 := \min\{t \ge 0 : x(t) \notin B_{2\delta}(x(0))\}, \quad t_2 := \inf\{t \ge t_1 : x(t) \in B_{2\delta}(x(0))\}.
\]
\end{lemma}

\begin{remark}\label{rm_tarj1a}
	The same proof applied to a backward trajectory of~\eqref{eq_ode1} with fixed $x(T)$ for some $T>0$ (instead of a forward trajectory with $x(0)$ fixed) shows that once
	\[
	|y_1-x(T)|\leq \delta^3
	\]
	there is a vector field $\tilde V\colon \R^d\to\R^d$ coinciding with $V$ outside of $B_{2\delta}(x(T))$ with
	\begin{equation*}\label{eq_tildeVest2b}
	\|V-\tilde V\|_{\Lip} < \varepsilon
	\end{equation*}
	such that the
	backward trajectory $y(\cdot)$ of
	\begin{equation*}
	\dot{y}=\tilde V(y)
	\end{equation*}
		satisfying $y(T)=y_1$ coincides with $x(\cdot)$ such that the coincides with $x(\cdot)$ over the segment $[t_1, t_2]$ defined above.
\end{remark}

\begin{remark}\label{rm_tarj1b}
	The same proof can be applied also to nonautonomous ODEs. For instance, in this way we prove that for every $\varepsilon>0$ there is a $\delta\in (0,1)$ such that if $x(\cdot)$ is a trajectory of
	\[
	\dot{x} = F(t,x),
	\]
	with a bounded vector field $F\colon \R\times \R^d\to \R^d$, 
	satisfying
	\[
	|F(t,x_1)|- F(t,x_2)|\leq L |x_1-x_2|, 
	\]
	then for every $y_0\in  \R^d$ 
	such that
	\[
	|y_0-x(0)|\leq \delta^3, 
	\]
	there is a vector field $\tilde F\colon \R^d\to\R^d$ coinciding with $F$ outside of 
$\R\times B_{2\delta}(x(0))$ 
	with
	\begin{equation*}\label{eq_tildeVestF1a}
	\|F-\tilde F\|_{\infty} < \varepsilon
	\end{equation*}
	such that the  trajectory $y(\cdot)$ of
	\begin{equation*}\label{eq_odeF2}
	\dot{y}=\tilde F(t,y)
	\end{equation*}
	satisfying $y(0)=y_0$ 
	coincides with $\theta$ outside of $B_{2\delta}(x(0))$. 
	Moreover, if 
	\[F(t,x)=V(x)+u(t),\] then $\tilde F(t,x)=\tilde V(x)+u(t)$, with
	\[\|V-\tilde V\|_{\infty} < \varepsilon.
	\] 
Of course, if one is interested only in the the smallness of $\|F-\tilde F\|_{\infty}$, only Steps~1-4 of the proof are needed.  
\end{remark}

\begin{proof}[Proof of Lemma~\ref{lm_corrTraj1}]
Consider a function $\varphi\in C_0^\infty(\R^d)$ with $\supp \varphi\in B_{2}(0)$ satisfying
\begin{align*}
0\leq \varphi(x)\leq 1 &\quad  \mbox{for all $x\in \R^d$},\\
\varphi(x)=1 &\quad  \mbox{for all $x\in \bar B_{1}(0)$},
\end{align*}
and set
\begin{align*}
\Phi(x) &:= \varphi_\delta(x) (x-(y_0-x(0))) + (1-\varphi_\delta(x)) x 
  ,\quad\mbox{where}\\
&\varphi_\delta(x) := \varphi\left(\frac{x-x(0)}{\delta}\right).
\end{align*}
By plugging $x=\Phi(y)$ into~\eqref{eq_ode1}, we get
\[
\dot{y}=(D\Phi)^{-1}(y(t)) V(\Phi(y(t))),
\]
which is~\eqref{eq_ode2} with
\begin{equation*}\label{eq_deftildV1}
\tilde V(y):=(D\Phi)^{-1}(y) V(\Phi(y)).
\end{equation*}

The rest of the proof will be divided in several steps.

{\sc Step 1} (preparatory observations).
We first note that
\[
\Phi(x) = x - \varphi_\delta (x) (y_0-x(0)).
\]
Therefore,
\begin{align*}
 D\Phi  & = \mbox{Id}- \nabla\varphi_\delta\otimes (y_0-x(0)),\\
 \partial_{y_j} D\Phi  & = - \nabla (\partial_{y_j}\varphi_\delta)\otimes (y_0-x(0)),
\end{align*}
so that, recalling
\begin{align*}
\|\nabla\varphi_\delta\|_\infty &\leq \frac{\|\nabla\varphi\|_\infty}{\delta},\\
\|\nabla (\partial_{y_j}\varphi_\delta)\|_\infty &\leq \frac{\|D^2\varphi\|_\infty}{\delta^2},
\end{align*}
we get
\begin{eqnarray}
\label{eq_estPhi1}
|y-\Phi(y)| &\leq \delta^3,\\
\label{eq_estDPhi1}
|D\Phi(y)|  & \geq 1-\|\nabla\varphi\|_\infty\delta^2,\quad\mbox{for all $y\in\R^d$},\\
\label{eq_estDPhi2}
\|D\Phi-\mbox{Id}|\|_\infty  & \leq \|\nabla\varphi\|_\infty\delta^2,\\
\label{eq_estD2Phi1}
\|\partial_{y_j} D\Phi\|_\infty  & = \|D^2\varphi\|_\infty\delta.
\end{eqnarray}
In particular, from~\eqref{eq_estDPhi1} we get that when
\begin{equation}\label{eq_deltaest1}
0\leq \delta < \frac{1}{\|\nabla\varphi\|_\infty^{1/2}}, 
\end{equation}
then $|D\Phi(y)|>0$ for all $y\in\R^d$.
Since $|\Phi(y)|\geq |y|-\delta^3$, we have that $\Phi\colon \R^d\to\R^d$ is a proper map (i.e. preimage of a compact set is precompact), and hence under the condition~\eqref{eq_deltaest1} it is globally invertible.  

{\sc Step 2}. Since $\Phi(x)=x$ for $x\in B_{2\delta}(x(0))^c$, we get that
$y(t)\in\theta$ when $y(t)\not \in B_{2\delta}(x(0))$. Further, under condition~\eqref{eq_deltaest1} since $\Phi$ is globally invertible,
then for every $z\in \theta\setminus B_{2\delta}(x(0))$ one has that $z=y(t)$ for some $t\in\R$, i.e.\ in other words
$y(\cdot)$ coincides with $\theta$ outside of $B_{2\delta}(x(0))^c$.

{\sc Step 3}. We claim that $y(0)=y_0$, if $\delta<1$. In fact, $x(0)=\Phi(y(0))$. This means that
$y(0)\in B_{\delta}(x(0))$, since 
\[
y(0)-x(0)=\varphi_\delta(y(0))(y_0-x(0)),
\]
and hence
\[
|y(0)-x(0)| \leq |y_0-x(0)|\leq \delta^3 <\delta. 
\]
But $\Phi(x)=x-(y_0-x(0))$ for $x\in \bar B_{\delta}(x(0))$, and hence
\[
x(0)=\Phi(y(0))=y(0)-(y_0-x(0)),
\]
implying the claim.

{\sc Step 4}. It remains to prove~\eqref{eq_tildeVest1} for a suitable choice of $\delta<1$. 
We prove first 
\begin{equation}\label{eq_tildeVestC0}
\|V-\tilde V\|_{\infty} < \varepsilon/2.
\end{equation}
To this aim we note
\begin{align*}
|V(y)-\tilde V(y)| &\leq  |V(y)-V(\Phi(y))| + |V(\Phi(y))-(D\Phi)^{-1}(y) V(\Phi(y))| \\
  &\leq \Lip V |y-\Phi(y)| + \|V\|_\infty |\mbox{Id}-(D\Phi)^{-1}(y)| \\
&\leq \Lip V |y-\Phi(y)| + \|V\|_\infty |(D\Phi)^{-1}(y)||(D\Phi)(y)-\mbox{Id}|.
\end{align*}
Plugging~\eqref{eq_estPhi1},~\eqref{eq_estDPhi1} and~\eqref{eq_estDPhi2} into the above estimate, we get
\begin{equation}\label{eq_tildeVestC0a}
|V(y)-\tilde V(y)| \leq \Lip V \delta^3  + \|V\|_\infty \frac{\|\nabla\varphi\|_\infty\delta^2}{1-\|\nabla\varphi\|_\infty\delta^2},
\end{equation}
so that to get~\eqref{eq_tildeVestC0} it is enough to take $\delta>0$ so that the right-hand side of~\eqref{eq_tildeVestC0a} be less than $\varepsilon/2$. 

{\sc Step 5}. To conclude the proof of~\eqref{eq_tildeVest1}, it remains to prove 
\begin{equation}\label{eq_tildeVestC1}
\|DV-D\tilde V\|_{\infty} < \varepsilon/2.
\end{equation}
Suppose $V\in \Lip(\R^d;\R^d)\cap C^1(\R^d;\R^d)$ with uniformly continous derivatives. 
We estimate
\begin{equation}\label{eq_tildeVest3}
\begin{aligned}
|DV(y)-D\tilde V(y)|\leq & |DV(y)-(D\tilde V)(\Phi(y))| \\
& \qquad + |D((D\Phi)^{-1})(y) \cdot V(\Phi(y))|.
\end{aligned}
\end{equation}
Clearly, when $\delta<1$, we have $x(0)\in B_1(y_0)$. 
Denoting by $\omega(\cdot)$ the (nondecreasing) modulus of continuity of $DV$ over $\bar B_1(y_0)$, we have 
\begin{equation}\label{eq_tildeVest3a}
|DV(y)-(D\tilde V)(\Phi(y))| \leq \omega(|y-\Phi(y)|)\leq \omega(\delta^3).
\end{equation}
Further, from
\[
\partial_{y_j}((D\Phi)^{-1})(y)= -(D\Phi)^{-1}(y) \partial_{y_j}(D\Phi)(y) (D\Phi)^{-1}(y)  
\]
and~\eqref{eq_estDPhi1},~\eqref{eq_estD2Phi1} we get
\begin{equation}\label{eq_tildeVest3b}
|D((D\Phi)^{-1})(y) \leq \frac{\|D^2\varphi\|_\infty\delta}{\left(1-\|\nabla\varphi\|_\infty\delta^2\right)^2}.
\end{equation}
Plugging~\eqref{eq_tildeVest3a} and~\eqref{eq_tildeVest3b} into~\eqref{eq_tildeVest3}, we obtain
\[
|DV(y)-D\tilde V(y)|\leq \omega(\delta^3) + \frac{\|D^2\varphi\|_\infty\delta}{\left(1-\|\nabla\varphi\|_\infty\delta^2\right)^2},
\]
which implies~\eqref{eq_tildeVestC1} and hence, toghether with~\eqref{eq_tildeVestC0}, also~\eqref{eq_tildeVest1}, thus concluding the proof.
\end{proof}

\end{document}